\documentclass[reqno]{amsart}

\usepackage{amsmath,amssymb,amsthm}
\usepackage{mathtools}\mathtoolsset{showonlyrefs}
\usepackage{hyperref}
\newcommand{\Id}{\mathrm{Id}}
\newcommand{\Wr}{\mathrm{Wr}}
\newcommand{\pd}{\partial}
\newcommand{\Rset}{\mathbb{R}}
\newcommand{\Nset}{\mathbb{N}}
\newcommand{\Zset}{\mathbb{Z}}
\newcommand{\Tset}{\mathbb{T}}

\newcommand{\Bt}{\tilde{\mathcal{B}}}
\newcommand{\Bh}{\hat{\mathcal{B}}}
\newcommand{\cI}{\mathcal{I}}
\newcommand{\cO}{\mathcal{O}}
\newcommand{\ep}{\varepsilon}
\newtheorem{theorem}{Theorem}[section]
\newtheorem{proposition}[theorem]{Proposition}
\newtheorem{lemma}[theorem]{Lemma}
\newtheorem{corollary}[theorem]{Corollary}
\theoremstyle{definition}
\newtheorem{definition}[theorem]{Definition}
\theoremstyle{remark}
\newtheorem{remark}[theorem]{Remark}
\numberwithin{equation}{section}
\begin{document}

\title{Darboux transformations from the Appell-Lauricella operator}

\author[A. M. Delgado]{Antonia M. Delgado$^1$}
\address[A. M. Delgado]{Dpto. Matem\'atica Aplicada, Universidad de Granada, 18071 Granada, Spain}
\email{amdelgado@ugr.es}

\author[L. Fern\'andez]{Lidia Fern\'andez$^1$}
\address[L. Fern\'andez]{Dpto. Matem\'atica Aplicada, Universidad de Granada, 18071 Granada, Spain}
\email{lidiafr@ugr.es}

\author[P. Iliev]{Plamen Iliev$^2$}
\address[P. Iliev]{School of Mathematics, Georgia Institute of Technology, Atlanta,  GA 30332-0160, USA}
\email{iliev@math.gatech.edu}
\thanks{$^1$AMD and LF are partially supported by Ministerio
de Ciencia, Innovaci\'on y Universidades of Spain (MICINN) and European Regional Development Fund (ERDF) through the grant PGC2018-094932-B-I00, and by Research Group FQM-384 from Junta de Andaluc\'ia.}
\thanks{$^2$PI is partially supported by Simons Foundation Grant \#635462.}

\date{September 17, 2019}



\begin{abstract}
We define two isomorphic algebras of differential operators: the first algebra consists of ordinary differential operators and contains the hypergeometric differential operator, while the second one consists of partial differential operators in $d$ variables and contains the Appell-Lauricella partial differential operator. Using this isomorphism, we construct partial differential operators which are Darboux transformations from polynomials of the Appell-Lauricella operator. We show that these operators can be embedded into commutative algebras of partial differential operators, containing $d$ mutually commuting and algebraically independent partial differential operators, which can be considered as quantum completely integrable systems. Moreover, these algebras can be simultaneously diagonalized on the space of polynomials leading to extensions of the Jacobi polynomials orthogonal with respect to the Dirichlet distribution on the simplex.
\end{abstract}

\maketitle


\section{Introduction}\label{se1}

The famous hypergeometric equation with three regular singular points at $0$, $1$ and $\infty$ can be written as
$$t(1-t)\frac{d^2 y}{d t^2}+[c-(a+b+1)t]\frac{d y}{d t}-aby=0,$$
where $a,b,c$ are parameters. Its analytic solution at $t=0$ is given by the Gauss hypergeometric function
\begin{equation*}
{}_2F_1\left( \begin{matrix}a,\,b \\ c \end{matrix} \,;\, t \right)=\sum_{k=0}^{\infty}\frac{(a)_k(b)_k}{k!\,(c)_k}t^k,
\end{equation*}
where $(a)_k$  denotes the Pochhammer symbol:
$$(a)_0=1\qquad \text{ and }\qquad (a)_k=a(a+1)\cdots(a+k-1)\text{ for } k\in\Nset.$$
If we set $a=-n$, $b=n+\alpha+\beta+1$, $c=\beta+1$, where $n\in\Nset_0$, the hypergeometric functions reduce to the Jacobi polynomials $p_n^{\alpha,\beta}(t) $ which are orthogonal with respect to the beta distribution. With this reparametrization, the hypergeometric equation can be rewritten as the eigenvalue equation
\begin{equation}\label{Jac_s}
\mathcal{M}_1^{\alpha,\beta} [ p_n^{\alpha,\beta}(t) ]= -n(n+\alpha+\beta+1)  p_n^{\alpha,\beta}(t),
\end{equation}
where $\mathcal{M}_1^{\alpha,\beta}$ is the hypergeometric (Jacobi) operator
\begin{equation}\label{Jacobi_op}
\mathcal{M}_1^{\alpha,\beta} = t(1-t)\partial_t^2 + [(\beta+1)-(\alpha+\beta+2)t]\partial_t.
\end{equation}
The spectral equation \eqref{Jac_s} plays a crucial role in numerous applications of the Jacobi polynomials in different branches of mathematics and physics.
Higher-order differential equations extending~\eqref{Jac_s} were built by Gr\"unbaum and Yakimov  \cite{GY02} by applying the general bispectral Darboux methods developed in \cite{BHY}. A different approach to these results was proposed in \cite{I11}, and it was used there to construct commutative algebras of partial differential operators invariant under rotations which are Darboux transformations from the partial differential operator of the classical orthogonal polynomials on the ball. The aim of this work is to construct analogous commutative algebras of partial differential operators which are Darboux transformations from the Appell-Lauricella operator, which we introduce below.

Multivariable extensions of the hypergeometric function ${}_2F_1$ have been developed by Appell \cite{A1881,A1882} in dimension $2$, and by Lauricella \cite{L1893} in arbitrary dimension. In particular, the Lauricella function $F_A$ defined by the equation
\begin{equation}\label{Lauricella_f}
\begin{split}
&F_A\left( \begin{matrix}a,\,b_1,\dots,b_d \\ c_1,\dots,c_d \end{matrix} \,;\, x_1,\dots,x_d \right)\\
&\qquad\qquad=\sum_{k_1,\dots,k_d=0}^{\infty}\frac{(a)_{k_1+\cdots+k_d}(b_1)_{k_1}\cdots(b_d)_{k_d}}{(c_1)_{k_1}\cdots (c_d)_{k_d}\,k_1!\cdots k_d!}\,x_1^{k_1}\cdots x_d^{k_d},
\end{split}
\end{equation}
can be characterized, up to an overall factor, as the unique analytic solution in a neighborhood of $(x_1,\dots,x_d)=(0,\dots,0)$ of the $d$ partial differential equations
\begin{equation}\label{Lau_PDEs}
\begin{split}
&x_k(1-x_k)\frac{\pd^2 F_A}{\pd x_k^2}-\sum_{\begin{subarray}{c}j=1 \\j\neq k\end{subarray}}^{d}x_kx_j\frac{\pd^2 F_A}{\pd x_k\pd x_j} \\
&\qquad\qquad+(c_k-(a+1+b_k)x_k)\frac{\pd F_A}{\pd x_k}-b_k\sum_{\begin{subarray}{c}j=1 \\j\neq k\end{subarray}}^{d}x_j\frac{\pd F_A}{\pd x_j}-ab_kF_A=0,
\end{split}
\end{equation}
where  $k=1,\dots,d$, see \cite{L1893}.
If we set $b_j=-\eta_j$, where $\eta_j\in\Nset_0$, $c_j=\gamma_j+1$ for $j=1,\dots,d$ and $a=\sum_{j=1}^{d}\eta_j+\sum_{j=1}^{d+1}\gamma_j+d$, it is clear that the Lauricella function in \eqref{Lauricella_f} becomes a polynomial $G_{\eta}(x;\gamma)$ in the variables $x_1,\dots,x_d$ of total degree $\eta_1+\cdots+\eta_d$ depending on the parameters $\gamma_1,\dots,\gamma_{d+1}$. Adding the differential equations \eqref{Lau_PDEs} satisfied by $F_A$, it follows that the polynomials  $G_{\eta}(x;\gamma)$ satisfy the spectral equation:
\begin{equation}\label{Laur_pde}
\mathcal{M}_{d}^\gamma [G_{\eta}(x;\gamma)] = -|\eta|\, (|\eta|+|\gamma| +d) \, G_{\eta}(x;\gamma),
\end{equation}
where $|\eta|=\eta_1++\cdots+\eta_d$, $|\gamma|= \gamma_1+\cdots+\gamma_{d+1}$ and
\begin{equation}\label{Laur_op}
\begin{split}
\mathcal{M}_{d}^\gamma =&
 \sum_{k=1}^d (1-x_k) x_k \partial^2_{x_k} -2 \sum_{1\leqslant k <l\leqslant d} x_k x_l \partial_{x_k}\partial_{x_l}
\\&+\sum_{k=1}^d [(\gamma_k+1) - (|\gamma|+d+1)x_k]\partial_{x_k}.
\end{split}
\end{equation}
In view of this, we refer to the operator in \eqref{Laur_op} as the {\em Appell-Lauricella operator}.
Clearly, when $d=1$, the Appell-Lauricella operator in \eqref{Laur_op} reduces to the hypergeometric operator~\eqref{Jacobi_op}, if we replace $(\alpha,\beta)$ by $(\gamma_2,\gamma_1)$, and thus, the spectral equation~\eqref{Laur_pde} can be considered as a natural multivariable extension of the hypergeometric equation. Moreover, one can show that the operator  $\mathcal{M}_{d}^\gamma$ is self-adjoint with respect to the Dirichlet distribution with parameters $\gamma_1,\dots,\gamma_{d+1}$ and, by an appropriate change of variables and gauge transformation it corresponds to the Hamiltonian for the generic quantum superintegrable system on the sphere \cite{KMT}.

In the present paper we construct Darboux transformations from specific polynomials of the Appell-Lauricella operator. It is perhaps useful to stress that there is an essential difference between univariate and multivariate Darboux transformations. In the univariate case,  up to an overall factor, a differential or difference operator can be uniquely determined from its kernel, which has dimension equal to the order of the operator. This implies that factorizations of univariate operators are essentially parametrized by subspaces of its kernel. The latter fact allows one to construct Darboux transformations by working with the kernel of the intertwining operator. However, all these constructions cannot be applied to partial differential operators, since the kernels are infinite dimensional, and factorizations and intertwining relations are much more subtle.

To overcome this difficulty, we use the ideas and techniques from \cite{I11} and we show that the hypergeometric differential operator \eqref{Jacobi_op} and the Appell-Lauricella operator \eqref{Laur_op} belong to two isomorphic associative algebras, denoted by $\mathfrak{D}_\alpha$ and $\hat{\mathfrak{D}}_\gamma$, respectively. Moreover, when we apply factorizations at one end of the spectrum of the recurrence operator for the Jacobi polynomials, the corresponding Darboux transformations  from the hypergeometric operator and intertwining operators also belong to the $\mathfrak{D}_\alpha$. Thus, we can use the isomorphism between $\mathfrak{D}_\alpha$ and $\hat{\mathfrak{D}}_\gamma$ to obtain multivariable Darboux transformations. The one-dimensional spectral equations depending on a parameter derived in \cite{I11} can be extended to this case to obtain an explicit basis of eigenfunctions for these operators in terms of appropriate extensions of the Jacobi polynomials on the simplex.

The paper is organized as follows. In the next section, we summarize several constructions and theorems established in \cite{I11} needed in the paper, together with some new one-dimensional results. In Section~\ref{se3}, we present the multivariable extensions. In the first subsection, we introduce the  notations and a brief account of the important multivariable ingredients. In particular, we define operators commuting with the Appell-Lauricella operator $\mathcal{M}_{d}^\gamma$, which can be simultaneously diagonalized by the Jacobi polynomials on the simplex, and we explain the connection to the generic superintegrable system on the sphere. We also construct the associative algebra $\hat{\mathfrak{D}}_\gamma$ which contains $\mathcal{M}_{d}^\gamma$ and the isomorphism between $\mathfrak{D}_\alpha$ and $\hat{\mathfrak{D}}_\gamma$. In the second subsection, we discuss extensions of these results for operators  obtained by  Darboux transformations from polynomials of the Appell-Lauricella operator and the corresponding quantum integrable systems. In Section~\ref{se4}, we illustrate the constructions in the paper with an explicit example, and we derive a Sobolev orthogonality relation for the associated polynomials.

\section{Extensions of the Jacobi polynomials and associated commutative algebras of differential operators}\label{se2}

In the  main results of the paper we shall use several constructions and theorems established in \cite{I11}, which we summarize in this section for the convenience of the reader, together with some new one-dimensional  results.

\subsection{Extensions of the Jacobi polynomials and recurrence relations}\label{ss2.1}
Throughout the paper, we shall use the classical Jacobi polynomials normalized as follows
\begin{equation}\label{pjacobi1}
p_n^{\alpha,\beta}(t) = (-1)^n \frac{(\alpha+\beta+1)_n}{n!}\frac{(\beta+1)_n}{(\alpha+1)_n} \, {_2}F_1\left( \begin{matrix} -n,\, n+\alpha+\beta+1 \\ \beta+1 \end{matrix} \,;\, t \right),
\end{equation}
which  are orthogonal with respect to the beta distribution
$$w_{\alpha,\beta}(t) =\frac{\Gamma(\alpha+\beta+2)}{\Gamma(\alpha+1)\Gamma(\beta+1)} (1-t)^\alpha t^\beta$$ on the interval $[0,1]$.
As we noted in the introduction, these polynomials are eigenfunctions of the hypergeometric operator, i.e.
\begin{equation}\label{2.3}
\mathcal{M}_1^{\alpha,\beta} [ p_n^{\alpha,\beta}(t) ]= -n(n+\alpha+\beta+1)  p_n^{\alpha,\beta}(t),
\end{equation}
where
\begin{equation}\label{difeq-1}
\mathcal{M}_1^{\alpha,\beta} = t(1-t)\partial_t^2 + [(\beta+1)-(\alpha+\beta+2)t]\partial_t.
\end{equation}
If we need to specify the variable of differentiation in the differential operator, we shall write $\mathcal{M}_1^{\alpha,\beta}(t)$ and we adopt this convention throughout the paper for all differential operators.
We set
\begin{equation}\label{lambda}
\lambda_n^{\xi} = -n(n+\xi),
\end{equation}
and therefore the eigenvalue in \eqref{2.3} is $\lambda_n^{\alpha+\beta+1}$.
\begin{remark}\label{re2.1}
Using Pfaff's identity \cite[page 79]{AAR} we can rewrite $p_n^{\alpha,\beta}(t) $ as follows:
\begin{equation}\label{2.2}
p_n^{\alpha,\beta}(t) = \frac{(\alpha+\beta+1)_n}{n!}
\, {_2}F_1\left( \begin{matrix} -n,\, n+\alpha+\beta+1 \\ \alpha+1 \end{matrix} \,;\, 1-t \right)
\end{equation}
which shows that $(-1)^np_n^{\alpha,\beta}(t)$ coincides with the Jacobi polynomial defined in \cite[equation~(2.1)]{I11} if we replace $t$ by $1-t$ and exchange the roles of $\alpha$ and $\beta$. Throughout the paper, we use this correspondence when we state results from \cite{I11}.
\end{remark}
It is well known that the Jacobi polynomials solve a {\em bispectral problem\/} in the sense of Duistermaat and Gr\"unbaum \cite{DG86}. Indeed, besides the spectral equation~\eqref{2.3}, the polynomials $p_n^{\alpha,\beta}(t)$ are also eigenfunctions of a difference operator acting on the degree index $n$.  More precisely, if we define the operator
\begin{equation}\label{eq:rec_op}
L_{\alpha,\beta}(n,E_n) = A_nE_n+B_n \, \Id+C_nE_n^{-1},
\end{equation}
where $E_n[f_n]=f_{n+1}$ is the shift operator acting on $n$, and the coefficients are given by
\begin{equation}\label{ABC}
\begin{aligned}
&A_n=\frac{(n+1)(n+\alpha+1)}{(2n+\alpha+\beta+1)(2n+\alpha+\beta+2)},
\\
&C_n=\frac{(n+\beta)(n+\alpha+\beta)}{(2n+\alpha+\beta)(2n+\alpha+\beta+1)},
\\
&B_n=A_n+C_n,
\end{aligned}
\end{equation}
then the Jacobi polynomials satisfy the three-term recurrence relation
\begin{equation}\label{eqn}
L_{\alpha,\beta}(n,E_n)\big[ p_n^{\alpha,\beta}(t)\big]=t\,p_n^{\alpha,\beta}(t).
\end{equation}

We refer to equations \eqref{2.3} and \eqref{eqn} as bispectral equations for the Jacobi polynomials. The connection between bispectrality and the Korteweg-de Vries hierarchy unraveled in \cite{DG86} suggested that soliton techniques can be used to construct extensions of the Jacobi polynomials which are eigenfunctions of higher-order differential and difference operators \cite{G2001,H2001}. In particular, following \cite{GY02}, we can define extensions of the  Jacobi polynomials which also satisfy bispectral equations as follows. We fix $k\in\Nset$ and we take $k$ arbitrary functions
\begin{equation}\label{eq:cond_psi}
\psi_n^{(0)},\dots,\psi_n^{(k-1)}\in (-1)^n \Rset[\lambda_n^{\alpha+\beta+1}],
\end{equation}
that is, each function is a polynomial of $\lambda_n^{\alpha+\beta+1}$ multiplied by $(-1)^n$. Using the collection $\psi=\big\{\psi_n^{(j)}, \  j=0,\dots,k-1\big\}$ we define new polynomials from the Jacobi polynomials {\em via} the formula
\begin{equation}\label{q}
q_n^{\alpha,\beta;\psi}(t) = \Wr_{n} \big(\psi_n^{(0)},\dots,\psi_n^{(k-1)}, p_n^{\alpha,\beta}(t)\big),
\end{equation}
where $\Wr_{n}\big(f_n^{(1)},\dots,f_n^{(k)}\big)=\det\big(f_{n-j+1}^{(i)}\big)_{1\leqslant i,j \leqslant k}$ denotes the discrete Wronskian. We assume that
\begin{equation}\label{tau}
\tau_n=\Wr_n\big((-1)^n\psi_n^{(0)},\dots,(-1)^n\psi_n^{(k-1)}\big)\neq 0,
\end{equation}
which is true for generic elements in $ (-1)^n \Rset[\lambda_n^{\alpha+\beta+1}]$. 
\begin{remark}\label{re:rec}
The importance of the above conditions can be explained as follows. We can consider a natural bi-infinite extension of the recurrence operator $L_{\alpha,\beta}(n,E_n)$ in \eqref{eq:rec_op} which acts on functions defined on $\Zset$. We can do this by making the formal change of variable $n\to n+\ep$ in the coefficients, where $\ep$ is a generic parameter and we denote the resulting operator by $L_{\alpha,\beta}(n+\ep,E_{n})$. With this convention and if \eqref{eq:cond_psi} holds, one can show that there exists a positive integer $m$ such that 
$$(L_{\alpha,\beta}(n+\ep,E_{n}))^m [\psi_{n+\ep}^{(j)}]=0 \qquad \text{ for }\qquad j=0,\dots,k-1,$$
i.e. the functions $\psi_{n+\ep}^{(j)}$ belong to the kernel of the  operator $(L_{\alpha,\beta}(n+\ep,E_{n}))^m$ for a sufficiently large power $m$. This explains the special choice of the functions $\psi_n^{(j)}$ and, in particular, the importance of the normalizing factor $(-1)^n$.  Using this fact, \eqref{eqn} and by considering the limit $\ep\to0$, one can show that the polynomials $q_n^{\alpha,\beta;\psi}(t) $ are eigenfunctions of difference operators acting on the degree index $n$ which extend the recurrence relation \eqref{eqn}. These operators are Darboux transformations from appropriate powers of the operator $L_{\alpha,\beta}(n,E_n)$, see \cite[Section 4]{I14} where this is proved in a more general setting within the context of the Askey-Wilson polynomials. In Section~\ref{ss4.1}, following \cite{GY02,I11}, we explain how we can pick the functions $\psi_n^{(j)}$ so that $q_n^{\alpha,\beta;\psi}(t)$ are eigenfunctions of a second-order difference operator which, by Favard's theorem, means that they will be orthogonal with respect to a nondegenerate moment functional.
\end{remark} 
In the next subsections, we describe the differential spectral equations for  $q_n^{\alpha,\beta;\psi}(t) $ which extend \eqref{2.3} for the Jacobi polynomials, the connection with the Darboux transformation and other results needed for the multivariable extensions.

\subsection{Commutative algebras of differential operators}\label{ss2.2}

In this subsection, we outline the construction of a commutative algebra of differential operators for which the polynomials $q_n^{\alpha,\beta;\psi}(t)$ are eigenfunctions.
We denote by $\mathfrak{D}_\alpha = \mathbb{R}\langle\mathcal{D}_1,\mathcal{D}_2\rangle$ the associative algebra generated by the differential operators $\mathcal{D}_1$ and $\mathcal{D}_2$, defined as
\begin{equation}\label{D1D2}
\mathcal{D}_1 = (t-1)\partial_t \quad \text{ and } \quad \mathcal{D}_2=(1-t)\partial_t^2 - (\alpha+1)\partial_t.
\end{equation}
It is easy to see that
\begin{equation}\label{W1}
[\mathcal{D}_2,\mathcal{D}_1]=\mathcal{D}_2\mathcal{D}_1-\mathcal{D}_1\mathcal{D}_2=\mathcal{D}_2,
\end{equation}
and also that the operator $\mathcal{M}_1^{\alpha,\beta}$ given in \eqref{difeq-1} can be expressed as a polynomial of $\mathcal{D}_1$ and $\mathcal{D}_2$ as follows
\begin{equation}\label{M1D1D2}
\mathcal{M}_1^{\alpha,\beta} =\mathcal{D}_2 -\mathcal{D}_1^2 -(\alpha+\beta+1)\mathcal{D}_1.
\end{equation}

One can show that, up to a simple factor, $\tau_n$ defined in \eqref{tau} is  a polynomial in $\lambda_{n-(k-1)/2}^{\alpha+\beta+1}$.  We denote by $\mathfrak{A}^{\alpha,\beta;\psi}$ the algebra of all polynomials $f$ such that $f\big(\lambda^{\alpha+\beta+1}_{n-k/2}\big)-f\big(\lambda^{\alpha+\beta+1}_{n-k/2-1}\big)$ is divisible by $\tau_{n-1}$ in $\mathbb{R}[n]$, that is,
$$
\mathfrak{A}^{\alpha,\beta;\psi} = \left\{f\in\mathbb{R}[t] \ : \ \frac{f\big(\lambda^{\alpha+\beta+1}_{n-k/2}\big)-f\big(\lambda^{\alpha+\beta+1}_{n-k/2-1}\big)}{\tau_{n-1}} \in \mathbb{R}[n] \right\}.
$$
With the above notations and using the convention in Remark~\ref{re2.1}, we can state Theorem 4.2 in \cite{I11} as follows.

\begin{theorem}\label{th1}
For every $f\in\mathfrak{A}^{\alpha,\beta;\psi}$, there exists a differential operator $\mathcal{B}_f = \mathcal{B}_f(\mathcal{D}_1,\mathcal{D}_2)\in\mathfrak{D}_\alpha$ such that
\begin{equation}\label{kr1}
\mathcal{B}_f \big[q^{\alpha,\beta;\psi}_n(t)\big] = f\big(\lambda^{\alpha+\beta+1}_{n-k/2}\big) q^{\alpha,\beta;\psi}_n(t).
\end{equation}
Moreover, $\mathfrak{D}^{\alpha,\beta;\psi}=\{\mathcal{B}_f \in \mathfrak{D}_\alpha : f\in \mathfrak{A}^{\alpha,\beta;\psi}\}$ is a commutative subalgebra of $\mathfrak{D}_\alpha$ which is isomorphic to $\mathfrak{A}^{\alpha,\beta;\psi}$.
\end{theorem}

From \cite[Lemma~4.5]{I11} we know that for every polynomial $r(n)\in\Rset[n]$ we can construct a differential operator $\Bt'_r =\Bt'_r (\mathcal{D}_1,\mathcal{D}_2)\in \mathfrak{D}_\alpha$ such that
\begin{equation}\label{N4}
r(n)p_{n}^{\alpha,\beta}(t)+r(-n-\alpha-\beta)p_{n-1}^{\alpha,\beta}(t)=\Bt'_r(\mathcal{D}_1,\mathcal{D}_2)[p_{n}^{\alpha,\beta}(t)+p_{n-1}^{\alpha,\beta}(t)].
\end{equation}
On the other hand, using the explicit formula~\eqref{pjacobi1} one can show that
\begin{align}
&p_{n}^{\alpha,\beta}(t)+p_{n-1}^{\alpha,\beta}(t)=\frac{2n+\alpha+\beta}{\alpha+\beta}p_{n}^{\alpha,\beta-1}(t),\label{N5}\\
&p_{n-1}^{\alpha,\beta}(t)=-\,\frac{1}{(\alpha+\beta)(\alpha+\beta-1)}\mathcal{D}_2[p_{n}^{\alpha,\beta-2}(t)].\label{N1}
\end{align}
Combining \eqref{N4} with \eqref{N5} we see that for every $r(n)\in\Rset[n]$ we can construct a differential operator $\Bt_r \in \mathfrak{D}_\alpha$ such that
\begin{equation}\label{N2}
r(n)p_{n}^{\alpha,\beta}(t)+r(-n-\alpha-\beta)p_{n-1}^{\alpha,\beta}(t)=(2n+\alpha+\beta)\Bt_r[p_{n}^{\alpha,\beta-1}(t)].
\end{equation}
Note also that from \eqref{M1D1D2} we have
\begin{equation}\label{N3}
\mathcal{M}_1^{\alpha,\beta+s} =\mathcal{D}_2 -\mathcal{D}_1^2 -(\alpha+\beta+s+1)\mathcal{D}_1\in  \mathfrak{D}_\alpha,
\end{equation}
for arbitrary $s\in\Rset$.

\begin{lemma}\label{lemm:conn}
For $l,j\in\mathbb{N}_0$ and $r(n)\in\Rset[n]$ we can construct differential operators $\Bh'_{r,l,j}\in\mathfrak{D}_\alpha $ and $\Bh''_{r,l,j}\in\mathfrak{D}_\alpha $ such that
\begin{align}
&r(n-l)p_{n-l}^{\alpha,\beta}(t)+(-1)^{j+1}r(-n-\alpha-\beta+l+j-1)p_{n-j-l}^{\alpha,\beta}(t)\nonumber\\
&\qquad=(2n+\alpha+\beta-2l-j+1)
\Bh'_{r,l,j}[p_{n}^{\alpha,\beta-j-2l}(t)],\label{K1}\\
\intertext{and}
&r(n-l)p_{n-l}^{\alpha,\beta}(t)+(-1)^{j}r(-n-\alpha-\beta+l+j-1)p_{n-j-l}^{\alpha,\beta}(t)\nonumber\\
&\qquad= \Bh''_{r,l,j}[p_{n}^{\alpha,\beta-j-2l}(t)].\label{K2}
\end{align}
\end{lemma}
\begin{proof}
First note that if \eqref{K1} holds for $l=0$, then we can prove it for all $l\in\Nset_0$ by using \eqref{N1}. The proof of \eqref{K1} now follows by induction on $j$, using equations~\eqref{2.3} and \eqref{N2}. Replacing $r(n)$ with $(2n+\alpha+\beta-j+1)r(n)$ in \eqref{K1} yields \eqref{K2}.
\end{proof}
Using the above lemma we can construct a differential analog of formula \eqref{q}.
\begin{proposition}\label{prop:conn}
Let
\begin{equation}\label{eq:DF}
c_{n,k}^{\alpha+\beta}=\begin{cases}
(-1)^{nk} & \text{ if }k\equiv 0 \text{ or } k\equiv 3 \mod 4,\\
(-1)^{nk}(2n+\alpha+\beta-k+1) & \text{ if }k\equiv 1 \text{ or } k\equiv 2 \mod 4.
\end{cases}
\end{equation}
We can construct a differential operator $\Bh_{\psi}=\Bh_{\psi}(\mathcal{D}_1,\mathcal{D}_2)\in\mathfrak{D}_\alpha$ such that
\begin{equation}\label{eq:conn}
q_n^{\alpha,\beta;\psi}(t)= c_{n,k}^{\alpha+\beta}\,\Bh_{\psi}(\mathcal{D}_1,\mathcal{D}_2)[p_{n}^{\alpha,\beta-k}(t)].
\end{equation}
\end{proposition}
\begin{proof}
We know that $\psi_n^{(j)}=(-1)^nf^{(j)}_n$, where $f^{(j)}_n\in \Rset[\lambda_n^{\alpha+\beta+1}]$ for $j=0,\dots,k-1$. If we substitute this into the right-hand side of equation~\eqref{q} and expand the determinant along the last column, we see that
\begin{equation}\label{conn1}
q_n^{\alpha,\beta;\psi}(t)=(-1)^{k(2n-k+1)/2}\sum_{l=0}^{k} \Delta_{n}^{(l)}p_{n-l}^{\alpha,\beta}(t),
\end{equation}
where $\Delta_{n}^{(l)}$ is the $k\times k$ determinant with $(i_1,i_2)$ entry $f^{(i_2)}_{n-i_1}$, and $i_1\in\{0,1,\dots,k\}\setminus\{l\}$,  $i_2\in\{0,1,\dots,k-1\}$.
If $\cI$ is the involution on $\Rset[n]$, defined by
$$\cI(n)=-(n-k+\alpha+\beta+1),$$
then it is easy to see that $\cI(f^{(j)}_{n-l})=f^{(j)}_{n-k+l}$. Therefore, applying $\cI$ to the determinant $\Delta_{n}^{(l)}$ gives $\Delta_{n}^{(k-l)}$ with reversed rows, hence
$$\cI(\Delta_{n}^{(l)})=(-1)^{k(k-1)/2}\Delta_{n}^{(k-l)}.$$
If we combine the terms $ \Delta_{n}^{(l)}p_{n-l}^{\alpha,\beta}(t)$ and  $\Delta_{n}^{(k-l)}p_{n-k+l}^{\alpha,\beta}(t)$ in \eqref{conn1} for $l=0,1,\dots,\lfloor\frac{k}{2}\rfloor$, we see that
\begin{align}
&\Delta_{n}^{(l)}p_{n-l}^{\alpha,\beta}(t)+\Delta_{n}^{(k-l)}p_{n-k+l}^{\alpha,\beta}(t)\nonumber\\
&\qquad=\Delta_{n}^{(l)}p_{n-l}^{\alpha,\beta}(t)+(-1)^{k(k-1)/2}\Delta_{-n+k-\alpha-\beta-1}^{(l)}p_{n-k+l}^{\alpha,\beta}(t)\nonumber\\
&\qquad=\Delta_{n}^{(l)}p_{n-l}^{\alpha,\beta}(t)+(-1)^{k+\ep}\Delta_{-n+k-\alpha-\beta-1}^{(l)}p_{n-k+l}^{\alpha,\beta}(t),\label{conn2}
\end{align}
where $\ep$ is the remainder of the division of $k(k-3)/2$ by $2$, i.e.
\begin{equation}\label{eps}
\ep =\begin{cases}
0 & \text{ if }k\equiv 0 \text{ or } k\equiv 3 \mod 4,\\
1 & \text{ if }k\equiv 1 \text{ or } k\equiv 2 \mod 4.
\end{cases}
\end{equation}
Applying now Lemma~\ref{lemm:conn} with $j=k-2l$ to \eqref{conn2}, we can construct a differential operator $\Bh_l \in \mathfrak{D}_\alpha$ such that
$$
(-1)^{nk}\left(\Delta_{n}^{(l)}p_{n-l}^{\alpha,\beta}(t)+\Delta_{n}^{(k-l)}p_{n-k+l}^{\alpha,\beta}(t)\right)=c_{n,k}^{\alpha+\beta} \,\Bh_{l}[p_{n}^{\alpha,\beta-k}(t)].
$$
The proof follows by rearranging the terms in \eqref{conn1} and by using the last equation.
\end{proof}

Using Proposition~\ref{prop:conn} we can relate the operators in $\mathcal{B}_f$ constructed in Theorem~\ref{th1} to the hypergeometric operator \eqref{difeq-1} by a Darboux transformation.

\begin{corollary} \label{co:1DDarboux}
For $f\in\mathfrak{A}^{\alpha,\beta;\psi}$, the operators $\mathcal{B}_f(\mathcal{D}_1,\mathcal{D}_2)$ constructed in  Theorem~\ref{th1} and $f(\mathcal{M}_1^{\alpha,\beta-k}+\lambda_{k/2}^{-\alpha-\beta-1})$ satisfy the intertwining relation
\begin{equation}\label{D1}
\mathcal{B}_f(\mathcal{D}_1,\mathcal{D}_2)\,\Bh_{\psi}(\mathcal{D}_1,\mathcal{D}_2)= \Bh_{\psi}(\mathcal{D}_1,\mathcal{D}_2)\, f(\mathcal{M}_1^{\alpha,\beta-k}+\lambda_{k/2}^{-\alpha-\beta-1}),
\end{equation}
where $\Bh_{\psi}$ is the operator defined in Proposition~\ref{prop:conn}.
\end{corollary}
\begin{proof}
It is easy to check that
\begin{equation}\label{eq:laid}
\lambda_{n-k/2}^{\alpha+\beta+1}=\lambda_{n}^{\alpha+\beta-k+1}+\lambda_{k/2}^{-\alpha-\beta-1}.
\end{equation}
Combining this with equations \eqref{2.3}, \eqref{eq:conn} and \eqref{kr1} we see that
\begin{align*}
c_{n,k}^{\alpha+\beta}\,\mathcal{B}_f\,\Bh_{\psi}[p_{n}^{\alpha,\beta-k}(t)] &=\mathcal{B}_f \big[q^{\alpha,\beta;\psi}_n(t)\big] \\
&=f\big(\lambda^{\alpha+\beta+1}_{n-k/2}\big) q^{\alpha,\beta;\psi}_n(t)\\
&=c_{n,k}^{\alpha+\beta}\,f\big(\lambda^{\alpha+\beta+1}_{n-k/2}\big)\,\Bh_{\psi}[p_{n}^{\alpha,\beta-k}(t)]\\
&=c_{n,k}^{\alpha+\beta}\,\Bh_{\psi}\,[f\big(\lambda_{n}^{\alpha+\beta-k+1}+\lambda_{k/2}^{-\alpha-\beta-1}\big) p_{n}^{\alpha,\beta-k}(t)]\\
&=c_{n,k}^{\alpha+\beta}\,\Bh_{\psi}\,[f\big(\mathcal{M}_1^{\alpha,\beta-k}+\lambda_{k/2}^{-\alpha-\beta-1}\big) p_{n}^{\alpha,\beta-k}(t)],
\end{align*}
which shows that
$$\left(\mathcal{B}_f\,\Bh_{\psi}- \Bh_{\psi}\, f(\mathcal{M}_1^{\alpha,\beta-k}+\lambda_{k/2}^{-\alpha-\beta-1})\right)[p_{n}^{\alpha,\beta-k}(t)]=0.$$
Since the last equation is true for all $n\in\Nset_0$, we conclude that \eqref{D1} holds, completing the proof.
\end{proof}

\begin{remark}
Recall that if two operators $\cO_1,\cO_2$ acting on the same space can be connected by the intertwining relation
\begin{equation}\label{Darbouxg}
\cO_{2}\cO=\cO\cO_1,
\end{equation}
with another operator $\cO$, then we say the operator $\cO_{2}$ is obtained
from $\cO_1$ by a Darboux transformation, after the work by G.~Darboux~\cite{D1882}. Thus,
Corollary~\ref{co:1DDarboux} shows that the commuting operators $\mathcal{B}_f$ are Darboux transformations from the commuting operators $\tilde{f}(\mathcal{M}_1^{\alpha,\beta-k})$, where $\tilde{f}(t)=f(t+\lambda_{k/2}^{-\alpha-\beta-1})$ and $\mathcal{M}_1^{\alpha,\beta-k}$ is the hypergeometric operator. We note that this also follows from the work of Gr\"unbaum and Yakimov \cite{GY02}. Our approach provides a new constructive proof of this fact, based on the techniques developed in \cite{I11}, which show that the operators $\mathcal{B}_f$ and the intertwining operator $\Bh_{\psi}$ belong to $\mathfrak{D}_\alpha$. The fact that $\mathcal{B}_f,\Bh_{\psi}\in\mathfrak{D}_\alpha$ will allow us to extend the Darboux transformation above to the multivariable setting.
\end{remark}

\subsection{Operators and polynomials depending on an additional parameter}\label{ss2.3}

Now, we introduce a modification of the differential operator $\mathcal{D}_2$ which is crucial for the multivariable extensions.
For arbitrary $s\in\mathbb{N}_0$, we define the operators
\begin{equation}\label{D2s}
\mathcal{D}_{2,s} = \mathcal{D}_2-\frac{s(s+\alpha)}{1-t},
\end{equation}
and
\begin{equation}\label{m1salg}
\mathcal{M}_{1}^{\alpha,\beta;s} = \mathcal{D}_{2,s}-\mathcal{D}_1^2 - (\alpha+\beta+1)\mathcal{D}_1.
\end{equation}
It is easy to see that
$$[\mathcal{D}_{2,s},\mathcal{D}_1]=\mathcal{D}_{2,s},$$
and
\begin{equation}\label{m1s}
\mathcal{M}_{1}^{\alpha,\beta;s}[ p_{n}^{\alpha+2s,\beta} (t)(1-t)^s] = \lambda_{n+s}^{\alpha+\beta+1} p_{n}^{\alpha+2s,\beta} (t)(1-t)^s.
\end{equation}

There is an extension of Proposition~\ref{prop:conn} which plays an important role in the multivariable case. Note first that the construction of the operator  $\Bh_{\psi}(\mathcal{D}_1,\mathcal{D}_{2})$ in equation~\eqref{eq:conn} depends only on the equations \eqref{N4}, \eqref{N5} and \eqref{N1}.
Using the remarks in \cite[page~457]{I11}, we see that equation~\eqref{N4} can be extended to
\begin{equation}\label{N4s}
\begin{split}
&r(n+s)p_{n}^{\alpha+2s,\beta}(t)(1-t)^{s}+r(-n-s-\alpha-\beta)p_{n-1}^{\alpha+2s,\beta}(t)(1-t)^{s}\\
&\qquad=\Bt'_r(\mathcal{D}_1,\mathcal{D}_{2,s})[p_{n}^{\alpha+2s,\beta}(t)(1-t)^{s}+p_{n-1}^{\alpha+2s,\beta}(t)(1-t)^{s}],
\end{split}
\end{equation}
where the operator $\Bt'_r(\mathcal{D}_1,\mathcal{D}_{2,s})$ is obtained from the operator $\Bt'_r(\mathcal{D}_1,\mathcal{D}_2)$ in \eqref{N4} by replacing $\mathcal{D}_2$ with $\mathcal{D}_{2,s}.$
Moreover, it is easy to see that equations \eqref{N5} and \eqref{N1} have the following extensions
\begin{align}
&p_{n}^{\alpha+2s,\beta}(t)(1-t)^s+p_{n-1}^{\alpha+2s,\beta}(t)(1-t)^s\nonumber\\
&\qquad\qquad=\frac{\alpha+\beta}{\alpha+\beta+2s}\frac{2(n+s)+\alpha+\beta}{\alpha+\beta}p_{n}^{\alpha+2s,\beta-1}(t)(1-t)^s,\label{N5s}\\
&p_{n-1}^{\alpha+2s,\beta}(t)(1-t)^s= -  \frac{(\alpha+\beta)(\alpha+\beta-1)}{(\alpha+\beta+2s)(\alpha+\beta+2s-1)}\nonumber\\
& \qquad\qquad \times
\frac{1}{(\alpha+\beta)(\alpha+\beta-1)} \mathcal{D}_{2,s}\left[p_{n}^{\alpha+2s,\beta-2}(t)(1-t)^s \right].\label{N1s}
\end{align}
Using equations \eqref{N4s}, \eqref{N5s}, \eqref{N1s} and adapting the proof of Proposition~\ref{prop:conn} we obtain the following extension.
\begin{proposition}\label{prop:Bhs}
Let
\begin{equation}\label{qcup}
\hat{q}_{n,s}^{\alpha,\beta;\psi}(t) = \Wr_{n} \big(\psi_{n+s}^{(0)},\dots,\psi_{n+s}^{(k-1)},p_n^{\alpha+2s,\beta}(t)\big),
\end{equation}
and
\begin{equation}\label{eq:chs}
\hat{c}_{n,k,s}^{\alpha+\beta}= c_{n+s,k}^{\alpha+\beta}\prod_{j=0}^{k-1}\frac{\alpha+\beta-j}{\alpha+\beta+2s-j},
\end{equation}
where $c_{n+s,k}^{\alpha+\beta}$ is defined in \eqref{eq:DF}.
Then we have
\begin{equation}\label{qgorros}
\hat{q}_{n,s}^{\alpha,\beta;\psi}(t)\,(1-t)^s=\hat{c}_{n,k,s}^{\alpha+\beta}\,\Bh_{\psi}(\mathcal{D}_1,\mathcal{D}_{2,s})[p_{n}^{\alpha+2s,\beta-k}(t)\,(1-t)^s],
\end{equation}
where $\Bh_{\psi}(\mathcal{D}_1,\mathcal{D}_{2,s})$ is obtained from the operator $\Bh_{\psi}(\mathcal{D}_1,\mathcal{D}_{2})$ in Proposition~\ref{prop:conn} by replacing $\mathcal{D}_2$ with $\mathcal{D}_{2,s}.$
\end{proposition}

\section{Multivariable operators and polynomials}\label{se3}

\subsection{Notations and preliminary results}\label{ss3.1}

We  start by introducing some vector notations which will be used in the rest of the paper.
For a vector $v=(v_1,\dots, v_r)$, we denote by $|v|=v_1+\dots+v_r$ the sum of its components. Moreover, for $1\leqslant j \leqslant r$, we define $\boldsymbol{v}_j = (v_1,\dots, v_j)$ and $\boldsymbol{v}^j = (v_j,\dots, v_r)$, with the convention $\boldsymbol{v}_0 =\boldsymbol{v}^{r+1}= 0$.

Throughout the paper, we consider  $\gamma=(\gamma_1,\dots,\gamma_{d+1})$ with components $\gamma_j>-1$ and $x=(x_1,\dots,x_d)\in\mathbb{R}^d$. The Dirichlet distribution on the simplex
$$\Tset^{d}=\{x\in\Rset^{d}:  x_i\geq 0\text{ and }|x|\leq 1\}$$
is defined by
$$
W(x) = \frac{\Gamma(|\gamma|+d+1)}{\prod_{j=1}^{d+1}\Gamma(\gamma_j+1)}\,x_1^{\gamma_1}\cdot\dots\cdot x_d^{\gamma_d} (1-|x|)^{\gamma_{d+1}}.
$$
A system of mutually orthogonal polynomials can be defined recursively (see \cite[Section~5.3]{DX14}) via the formula
\begin{equation}\label{simplexp}
P_\eta^\gamma(x) = p_{\eta_1}^{a_1,b_1}(z_1) \, (1-z_1)^{\eta_2+\dots+\eta_d} \, P_{\eta_2,\dots,\eta_d}^{\gamma_2,\dots,\gamma_{d+1}}(z_2,\dots,z_d),
\end{equation}
where
$\eta=(\eta_1,\dots,\eta_d)\in\mathbb{N}_0^d$ are indices,
$p_{\eta_1}^{a_1,b_1}(z_1)$ is the univariate Jacobi orthogonal polynomial defined in \eqref{pjacobi1}, with parameters
$$\begin{aligned}
& a_1=2(\eta_2+\cdots+\eta_d) + (\gamma_2+\cdots+\gamma_{d+1}) + d-1,
\\
& b_1=\gamma_1,
\end{aligned}
$$
$P_{\eta_2,\dots,\eta_d}^{\gamma_2,\dots,\gamma_{d+1}}(z_2,\dots,z_d)$ are the orthogonal polynomials on the $(d-1)$--dimensional simplex,
and the variables $x$ and $z$ are related as follows
\begin{equation}\label{chvar}
\begin{aligned}
& z_1=x_1,\\
& z_j=\frac{x_j}{1-x_1}, \quad\text{ for }\quad j=2,\dots,d.
\end{aligned}
\end{equation}
The polynomials $\{P_\eta^\gamma(x), \eta\in\mathbb{N}_0^d\}$ can be characterized as common eigenfunctions of $d$ commuting partial differential operators. More precisely, if we define
\begin{equation}\label{diffeq}
\begin{aligned}
\mathcal{M}_{j,d}^\gamma =&  \sum_{k=j}^d (1-|\boldsymbol{x}_{j-1}|-x_k) x_k \partial^2_{x_k} -2 \sum_{j\leqslant k <l\leqslant d} x_k x_l \partial_{x_k}\partial_{x_l}
\\
&+\sum_{k=j}^d [(\gamma_k+1)(1-|\boldsymbol{x}_{j-1}|) - (|{\boldsymbol{\gamma}}^j|+d-j+2)x_k]\partial_{x_k},
\end{aligned}
\end{equation}
then
\begin{equation}\label{speceq}
\mathcal{M}_{j,d}^\gamma [P_\eta^\gamma(x)] = \lambda_{|\boldsymbol{\eta}^j|}^{|\boldsymbol{\gamma}^j| +d+1-j} P_\eta^\gamma(x), \quad\text{ for }\quad  j=1,\dots,d,
\end{equation}
where the eigenvalue $\lambda_{|\boldsymbol{\eta}^j|}^{|\boldsymbol{\gamma}^j| +d+1-j} = - |\boldsymbol{\eta}^j| (|\boldsymbol{\eta}^j| + |\boldsymbol{\gamma}^j| +d+1-j )$ is defined in \eqref{lambda}, see \cite[Proposition~3.3]{I17}.

\begin{remark}\label{rem:ind}
Note that the operator $\mathcal{M}_{j,d}^\gamma$ is independent of the parameters $\gamma_{1},\dots,\gamma_{j-1}$.
\end{remark}
\begin{remark}\label{rem:Lau}
For $j=1$ we obtain the {\em Appell-Lauricella operator}  $\mathcal{M}_{1,d}^\gamma=\mathcal{M}_{d}^\gamma$ discussed in the introduction
\begin{equation}\label{diffeq1}
\begin{aligned}
\mathcal{M}_{d}^\gamma =&  \sum_{k=1}^d (1-x_k) x_k \partial^2_{x_k} -2 \sum_{1\leqslant k <l\leqslant d} x_k x_l \partial_{x_k}\partial_{x_l}
\\
&+\sum_{k=1}^d [(\gamma_k+1) - (|\gamma|+d+1)x_k]\partial_{x_k}.
\end{aligned}
\end{equation}
This operator has been extensively studied in the literature. When $j=1$, the eigenvalue in \eqref{speceq} depends only on the total degree of the polynomial $P_\eta^\gamma(x)$. Therefore, if $\mathcal{V}_n^d$ denotes the space of polynomials of total degree $n$ orthogonal to all polynomials of degree at most $n-1$ with respect to $W(x)$, then for every $R(x)\in \mathcal{V}_n^d$ we have
\begin{equation}\label{speceq2}
\mathcal{M}_{d}^\gamma [R(x)] = -n (n + |{\gamma}| +d) \, R(x).
\end{equation}
In dimension 2, the first (nonorthogonal) bases of $\mathcal{V}_n^2$ were constructed by Appell \cite{A1881} in terms of the hypergeometric functions $F_2$, and equation~\eqref{speceq2} can be deduced by adding the differential equations satisfied by $F_2$. In 1893, Lauricella \cite{L1893} extended the hypergeometric functions introduced by Appell \cite{A1882} and derived the partial differential equations satisfied by them. Biorthogonal bases of $\mathcal{V}_n^d$ for arbitrary $d$, which extend the constructions of Appell, can be defined in terms of the Lauricella functions $F_A$ and the operator $\mathcal{M}_{d}^\gamma$ can be obtained by adding the differential equations satisfied by $F_A$, see \cite{KMT}. Equation~\eqref{speceq2} also explains why the operator $\mathcal{M}_{2}^\gamma$ naturally appears in the  Krall-Sheffer classification of bivariate analogs of the classical orthogonal polynomials \cite{KS67}; see also \cite{DX14,Su99}.
\end{remark}

\begin{remark}\label{rem:Int1}
By changing the variables $x_j=y_j^2$, and after an appropriate gauge transformation, up to an additive constant, the operator $4\mathcal{M}_{d}^\gamma$ can be rewritten as
\begin{equation}\label{H1}
\mathcal{H}=\Delta_{\mathbb{S}^{d}}+V(y),
\end{equation}
where
$\Delta_{\mathbb{S}^{d}}$ is the Laplace-Beltrami operator on the sphere
$$\mathbb{S}^{d}=\{y\in\mathbb{R}^{d+1}:y_1^2+\cdots+y_{d+1}^2=1\}$$
with potential
$$V(y)=\frac{1}{4}\sum_{k=1}^{d+1}\frac{1-4\gamma_k^2}{y_k^2},$$
see \cite{KMT}. Under the same change of variables and gauge transformation, the operators $\mathcal{M}_{j,d}^\gamma$, for $j=2,\dots,d$ together with $\mathcal{H}$ provide $d$ algebraically independent and mutually commuting integrals of motion, thus showing that the system is completely integrable. Moreover, one can show that operators
\begin{align}
&\hat{\mathcal{H}}_{i,j}=x_ix_j(\pd_{x_i}-\pd_{x_j})^2+[(\gamma_i+1)x_j-(\gamma_j+1)x_i](\pd_{x_i}-\pd_{x_j}), \nonumber\\
&\hskip7cm  \text{ if }1\leq i<j\leq d,\label{H2}\\
\intertext{and}
&\hat{\mathcal{H}}_{j,d+1}=x_j(1-|x|)\pd_{x_j}^2+[(\gamma_j+1)(1-|x|)-(\gamma_{d+1}+1)x_j]\pd_{x_j}, \nonumber\\
&\hskip7cm  \text{ if }1\leq j\leq d,\label{H3}
\end{align}
also commute with $\mathcal{M}_{d}^\gamma$, and therefore if we apply the same change of variables and gauge transformation to them, they will also provide integrals of motion for $\mathcal{H}$. Recall that classical or quantum Hamiltonian systems, possessing more than $d$ algebraically independent integrals of motion are usually referred to as superintegrable systems \cite{MPW}.
The system with Hamiltonian $\mathcal{H}$ in \eqref{H1} has been extensively studied in the literature as an important example of a second-order superintegrable system, possessing the maximal possible number of algebraically independent second-order integrals of motion. It is usually referred to as the {\em generic quantum superintegrable system on the sphere}, and has attracted a lot of attention recently in connection to multivariate extensions of the Askey scheme of hypergeometric orthogonal polynomials and their bispectral properties, the Racah problem for $\mathfrak{su}(1,1)$, representations of the Kohno-Drinfeld algebra, the Laplace-Dunkl operator associated with $\mathbb{Z}_2^{d+1}$ root system; see for instance \cite{DGVV,I17,KMP2} and the references therein. The space $\mathcal{V}_n^d$ introduced in Remark~\ref{rem:Int1} appears naturally in the analysis as an irreducible module over the associative algebra generated by the integrals of motion, see \cite{I18}.
\end{remark}

Using the change of variables \eqref{chvar} we can decompose the operator $\mathcal{M}_d^\gamma (x)$ as follows
$$
\mathcal{M}_d^\gamma (x) = \mathcal{M}_1^{|\boldsymbol\gamma^2|+d-1,\gamma_1}(z_1) + \frac{1}{1-z_1} \mathcal{M}_{d-1}^{\gamma_2,\ldots,\gamma_{d+1}} (z_2,\dots,z_d).
$$
Furthermore, using the identity in \cite[p. 2037]{I17}, we see that
\begin{equation}\label{eq:meqm}
\mathcal{M}_{2,d}^\gamma(x)=\mathcal{M}_{1,d-1}^{\gamma_2,\ldots,\gamma_{d+1}}(z_2,\dots,z_d) = \mathcal{M}_{d-1}^{\gamma_2,\ldots,\gamma_{d+1}}(z_2,\dots,z_d).
\end{equation}

Next, we define the partial differential operators
\begin{align}
\hat{D}_1 =& (x_1-1)\partial_{x_1} + \sum_{j=2}^d x_j\partial_{x_j},\label{eq:Dh1}
\\
\hat{D}_2 =& (1-x_1)\partial_{x_1}^2 + \sum_{j=2}^d x_j\partial_{x_j}^2 - 2 \sum_{j=2}^d (x_j \partial_{x_j})\partial_{x_1}\nonumber
\\
&
- (|\boldsymbol\gamma^2|+d) \partial_{x_1}  + \sum_{j=2}^d (\gamma_j+1)\partial_{x_j},\label{eq:Dh2}
\end{align}
which can be considered as multivariable extensions of the operators $\mathcal{D}_1$ and $\mathcal{D}_2$ defined in \eqref{D1D2},
and we denote by
$$\hat{\mathfrak{D}}_\gamma = \mathbb{R}\langle\hat{D}_1,\hat{D}_2\rangle$$
the associative algebra generated by the differential operators $\hat{D}_1$ and $\hat{D}_2$.
It is straightforward to check  that these operators satisfy the commutativity relation
$$[\hat{D}_2,\hat{D}_1]=\hat{D}_2$$
which combined with \eqref{W1} suggests defining the map
\begin{equation}\label{map}
\Phi:\quad \begin{matrix} \mathcal{D}_1\to \hat{D}_1\;\\  \mathcal{D}_2\to \hat{D}_2.
\end{matrix}
\end{equation}
It is not hard to see that $\Phi$ extends to an isomorphism from $\mathfrak{D}_\alpha$ onto $\hat{\mathfrak{D}}_\gamma$. This isomorphism depends on the parameters $\alpha,\gamma_2,\gamma_3,\dots,\gamma_{d+1}$, but since the parameters will be usually fixed and will not cause any confusion, we will omit this explicit dependence.

With these notations, a straightforward computation shows that the Appell-Lauricella operator $\mathcal{M}_d^\gamma$ defined in \eqref{diffeq1} can be written in terms of the operators $\hat{D}_1$ and $\hat{D}_2$ as follows
\begin{equation}\label{eq:ALop}
\mathcal{M}_d^\gamma=\hat{D}_2- \hat{D}_1^2 - (|\gamma|+d)\hat{D}_1.
\end{equation}
Combining the last equation with \eqref{M1D1D2} we see that if the parameters satisfy the relation
\begin{equation}\label{eq:condab}
\alpha+\beta+1=|\gamma|+d,
\end{equation}
then $\Phi$ maps the (one-variable) hypergeometric operator $\mathcal{M}_1^{\alpha,\beta}$ into the Appell-Lauricella partial differential operator $\mathcal{M}_d^\gamma$, i.e.
\begin{equation}\label{eq:JtoAL}
\Phi(\mathcal{M}_1^{\alpha,\beta})=\mathcal{M}_d^\gamma.
\end{equation}

The main point now is that this map can be extended to the Krall commutative algebras  $\mathfrak{D}^{\alpha,\beta;\psi}$ constructed in Theorem~\ref{th1}. This leads to commutative algebras of partial differential operators which are Darboux transformations from polynomials of the Appell-Lauricella operator.  Moreover, the operators in these algebras will commute with the operators $\{\mathcal{M}_{j,d}^\gamma\}_{j\geq 2}$ defined in \eqref{diffeq} and all these operators can be simultaneously diagonalized on the space of polynomials in terms of extensions of the Jacobi polynomials on the simplex.

\subsection{Darboux transformations}\label{ss3.2}
Applying the isomorphism $\Phi$ to the commutative algebra consisting of the ordinary differential operators $\mathcal{B}_f(\mathcal{D}_1,\mathcal{D}_2)$ constructed in Theorem~\ref{th1} we obtain a commutative algebra consisting of partial differential operators $\mathcal{B}_f(\hat{D}_1,\hat{D}_2)$.
From now on, we will assume that the parameters $(\alpha,\beta)$ and  $\gamma$ are related as follows:
\begin{equation}\label{eq:param}
\begin{split}
&\alpha=\gamma_2+\cdots+\gamma_{d+1}+d-1,\\
&\beta=\gamma_1.
\end{split}
\end{equation}
Note that \eqref{eq:condab} is satisfied and therefore, equation~\eqref{eq:JtoAL} also holds. Applying the isomorphism $\Phi$ to \eqref{D1} we obtain the following multivariable extension of Corollary~\ref{co:1DDarboux}.

\begin{proposition} \label{pr:multiDarboux}
With $\alpha$ defined in \eqref{eq:param}, and for $f\in\mathfrak{A}^{\alpha,\gamma_1;\psi}$, the operators $\mathcal{B}_f(\hat{D}_1,\hat{D}_2)$ and $f(\mathcal{M}_d^{\gamma-k \mathbf{e}_1}+\lambda_{k/2}^{-|\gamma|-d})$ satisfy the intertwining relation
\begin{equation}\label{DDD1}
\mathcal{B}_f(\hat{D}_1,\hat{D}_2)\,\hat{\mathcal{B}}_{\psi}(\hat{D}_1,\hat{D}_2)= \hat{\mathcal{B}}_{\psi}(\hat{D}_1,\hat{D}_2)\, f(\mathcal{M}_d^{\gamma-k \mathbf{e}_1}+\lambda_{k/2}^{-|\gamma|-d}),
\end{equation}
where $\hat{\mathcal{B}}_{\psi}(\hat{D}_1,\hat{D}_2)$ is the image under $\Phi$ of the operator $\hat{\mathcal{B}}_{\psi}(\mathcal{D}_1,\mathcal{D}_2)$ constructed in Proposition~\ref{prop:conn}, and $\mathbf{e}_1=(1,0,\dots,0)$.
\end{proposition}
Next, we will show that multivariable polynomials which diagonalize the operators $\mathcal{B}_f(\hat{D}_1,\hat{D}_2)$ in the last proposition can be defined in terms of appropriate extensions of the multivariable Jacobi polynomials on the simplex.
\begin{definition}\label{def:newP}
For $\eta\in\Nset_0^d$, $x\in\Rset^d$ and parameters $\gamma\in\Rset^{d+1}$ we define polynomials
\begin{equation}\label{Qeta}
Q_\eta^\gamma (x) = \hat{q}_{\eta_1,\eta_2+\cdots+\eta_d}^{\alpha,\beta;\psi} (z_1) (1-z_1)^{\eta_2+\cdots+\eta_d} P_{\eta_2,\dots,\eta_d}^{\gamma_2,\dots,\gamma_{d+1}}(z_2,\dots,z_d),
\end{equation}
where the variables $x$ and $z$ are related as in \eqref{chvar}, the polynomials $\hat{q}_{\eta_1,\eta_2+\dots+\eta_d}^{\alpha,\beta;\psi}$ are defined in \eqref{qcup}, and the parameters $\alpha,\beta$ are given in \eqref{eq:param}.
\end{definition}

\begin{proposition}\label{prop:BhsMulti}
The polynomials $Q_\eta^\gamma(x)$ in Definition~\ref{def:newP} can be related to the Jacobi polynomials on the simplex~\eqref{simplexp} via the formula
\begin{equation}\label{qgorros2}
Q_\eta^\gamma(x) = \hat{c}^{|\gamma|+d-1}_{\eta_1,k,s}\,\hat{\mathcal{B}}_\psi(\hat{D}_1,\hat{D}_2) [ P_\eta^{\gamma-k\mathbf{e}_1}(x)],
\end{equation}
where $s=|\boldsymbol\eta^2|=\eta_2+\cdots+\eta_d$ and  $\hat{c}^{|\gamma|+d-1}_{\eta_1,k,s}$ is defined in \eqref{eq:chs}
\end{proposition}
\begin{proof}
With $\alpha=|\boldsymbol\gamma^2|+d-1$, the differential operators in \eqref{D1D2} take the form
\begin{equation}
\mathcal{D}_1 =\mathcal{D}_1(z_1)= (z_1-1)\partial_{z_1} \quad \text{ and } \quad \mathcal{D}_2=\mathcal{D}_2(z_1)=(1-z_1)\partial_{z_1}^2 - (|\boldsymbol\gamma^2|+d)\partial_{z_1}.
\end{equation}
First, we see how the operators $\hat{D}_1$ and $\hat{D}_2$ act on polynomials $P(x)$ which, via the change of variables \eqref{chvar},  can be factored in separated variables as
$$P(x)=p(z_1)q(z_2,\dots,z_d).$$
A straightforward computation shows that
\begin{equation}\label{d1fact}
\hat{D}_1[P(x)] = \mathcal{D}_1[p(z_1)] q(z_2,\dots,z_d)
\end{equation}
and
\begin{equation}\label{d2fact}
\hat{D}_2[P(x)] =
\mathcal{D}_2[p(z_1)] q(z_2,\dots,z_d)
+ \frac{p(z_1)}{1-z_1}\mathcal{M}_{d-1}^{\gamma_2,\ldots,\gamma_{d+1}}[q(z_2,\dots,z_d)],
\end{equation}
where $\mathcal{M}_{d-1}^{\gamma_2,\ldots,\gamma_{d+1}}=\mathcal{M}_{d-1}^{\gamma_2,\ldots,\gamma_{d+1}}(z_2\dots,z_d)$ is the operator in \eqref{diffeq1} for the $(d-1)$ dimensional simplex in the variables $z_2,\dots,z_d$.
Applying the last formula to the polynomial $P_\eta^{\gamma-k\mathbf{e}_1} (x)$ defined in \eqref{simplexp} we obtain
\begin{equation}\label{pr1}
\begin{aligned}
\hat{D}_2 & [P_\eta^{\gamma-k\mathbf{e}_1} (x) ]  =
\mathcal{D}_2[p_{\eta_1}^{a_1,b_1-k} (z_1) (1-z_1)^s ] P_{\eta_2,\dots,\eta_d}^{\gamma_2,\dots,\gamma_{d+1}}(z_2,\dots,z_d)
\\
&+ \frac{p_{\eta_1}^{a_1,b_1-k}  (z_1) (1-z_1)^s}{1-z_1}\mathcal{M}_{d-1}^{\gamma_2,\ldots,\gamma_{d+1}}[P_{\eta_2,\dots,\eta_d}^{\gamma_2,\dots,\gamma_{d+1}}(z_2,\dots,z_d)].
\end{aligned}
\end{equation}
Equation \eqref{speceq} with $j=1$ and $d$ replaced by $d-1$ yields
\begin{equation}\label{pr2}
\mathcal{M}_{d-1}^{\gamma_2,\ldots,\gamma_{d+1}}[P_{\eta_2,\dots,\eta_d}^{\gamma_2,\dots,\gamma_{d+1}}(z_2,\dots,z_d)] = \lambda_{|\boldsymbol{\eta}^2|}^{|\boldsymbol{\gamma}^2| +d-1}P_{\eta_2,\dots,\eta_d}^{\gamma_2,\dots,\gamma_{d+1}}(z_2,\dots,z_d).
\end{equation}
Since $|\boldsymbol{\eta}^2|=s$ and $\lambda_{s}^{|\boldsymbol{\gamma}^2| +d-1} = -s(s+\gamma_2+\dots+\gamma_{d+1}+d-1)=-s(s+\alpha)$, we can see from \eqref{pr1}, \eqref{pr2} and \eqref{D2s} that
$$
\hat{D}_2 [P_\eta^{\gamma-k\mathbf{e}_1} (x)]  = \mathcal{D}_{2,s}[p_{\eta_1}^{a_1,b_1-k} (z_1) (1-z_1)^s ] P_{\eta_2,\dots,\eta_d}^{\gamma_2,\dots,\gamma_{d+1}}(z_2,\dots,z_d).
$$
The proof now follows from Proposition~\ref{prop:Bhs}.
 \end{proof}

\begin{theorem}\label{th:q}
The polynomials $\{Q_\eta^\gamma (x),\eta\in\mathbb{N}_0^d \}$ form a basis for the space $\Rset[x_1,\dots,x_d]$ of polynomials in $d$ variables. Furthermore, for $f\in\mathfrak{A}^{\alpha,\beta;\psi}$ we have
\begin{equation}\label{eq:Qeta}
\mathcal{B}_f(\hat{D}_1,\hat{D}_{2}) [Q_{\eta}^\gamma(x)] = f\big(\lambda^{|\gamma|+d}_{|\eta|-k/2}\big)Q_{\eta}^\gamma(x),
\end{equation}
where $\mathcal{B}_f(\hat{D}_1,\hat{D}_{2})$ is obtained from the operator $\mathcal{B}_f(\mathcal{D}_1,\mathcal{D}_2)$ in Theorem~\ref{th1}
by applying the map \eqref{map}, and $\lambda^{|\gamma|+d}_{|\eta|-k/2}$ is given in \eqref{lambda}.
\end{theorem}
\begin{proof}
The linear independence of the polynomials $\{Q_\eta^\gamma (x),\eta\in\mathbb{N}_0^d \}$ follows easily from the independence of the Jacobi polynomials on the simplex and the fact that $\hat{q}_{\eta_1,s}^{\alpha,\beta;\psi} (z_1)$ is a polynomial of degree $\eta_1$ in $z_1$. From \eqref{Qeta} we see that the total degree of the polynomial $Q_\eta^\gamma (x)$ is $|\eta|$. Therefore, the number of $Q_\eta^\gamma (x)$ of total degree $n$ is equal to $\binom{n+d-1}{d-1}$, which coincides with the number of monomials of total degree $n$, thus proving that the polynomials $\{Q_\eta^\gamma (x),\eta\in\mathbb{N}_0^d \}$ form a basis of the space $\Rset[x_1,\dots,x_d]$.

To establish equation~\eqref{eq:Qeta} we use consecutively equations \eqref{qgorros2}, \eqref{DDD1}, \eqref{speceq} with $j=1$, \eqref{eq:laid}, and \eqref{qgorros2} again:
\begin{align*}
\mathcal{B}_f(\hat{D}_1,\hat{D}_{2}) [Q_{\eta}^\gamma(x)]&= \hat{c}^{|\gamma|+d-1}_{\eta_1,k,s}\,
\mathcal{B}_f(\hat{D}_1,\hat{D}_{2})
\hat{\mathcal{B}}_\psi(\hat{D}_1,\hat{D}_2) [ P_\eta^{\gamma-k\mathbf{e}_1}(x)]\\
&= \hat{c}^{|\gamma|+d-1}_{\eta_1,k,s}\,
\hat{\mathcal{B}}_{\psi}(\hat{D}_1,\hat{D}_2)\, f(\mathcal{M}_d^{\gamma-k \mathbf{e}_1}+\lambda_{k/2}^{-|\gamma|-d}) [ P_\eta^{\gamma-k\mathbf{e}_1}(x)]\\
&= \hat{c}^{|\gamma|+d-1}_{\eta_1,k,s}\,
\hat{\mathcal{B}}_{\psi}(\hat{D}_1,\hat{D}_2)\, [f(\lambda_{|\eta|}^{|\gamma|+d-k}+\lambda_{k/2}^{-|\gamma|-d})  P_\eta^{\gamma-k\mathbf{e}_1}(x)]\\
&=f\big(\lambda^{|\gamma|+d}_{|\eta|-k/2}\big)\, \hat{c}^{|\gamma|+d-1}_{\eta_1,k,s}\,\hat{\mathcal{B}}_{\psi}(\hat{D}_1,\hat{D}_2)\, [ P_\eta^{\gamma-k\mathbf{e}_1}(x)]\\
&= f\big(\lambda^{|\gamma|+d}_{|\eta|-k/2}\big)Q_{\eta}^\gamma(x).
\end{align*}
\end{proof}

\begin{remark}
Note that similarly to the Appell-Lauricella equation~\eqref{speceq2}, the eigenvalue in \eqref{eq:Qeta} depends only on the total degree $|\eta|$ of the polynomials $Q_\eta^\gamma (x)$.
\end{remark}

\begin{lemma}\label{le:comm}
For $j\geq 2$, the operator $\mathcal{M}_{j,d}^{\gamma}$ defined in \eqref{diffeq} commutes with the operators in the algebra $\hat{\mathfrak{D}}_\gamma$.
\end{lemma}
\begin{proof}
Since $\hat{\mathfrak{D}}_\gamma$ is generated by the operators $\hat{D}_1$ and $\hat{D}_{2}$ in \eqref{eq:Dh1}-\eqref{eq:Dh2}, it is enough to show that for $j\geq 2$ we have
\begin{align}
[\mathcal{M}_{j,d}^{\gamma},\hat{D}_1]=0,\label{eq:commD1}\\
[\mathcal{M}_{j,d}^{\gamma},\hat{D}_2]=0.\label{eq:commD2}
\end{align}
It is easy to see that under the change of variables \eqref{chvar}, the operator $\hat{D}_1$ becomes $(z_1-1)\pd_{z_1}$, while the operator $\mathcal{M}_{j,d}^{\gamma}$ becomes a partial differential operator in the variables $z_2,\dots,z_d$ with coefficients independent of $z_1$, which proves \eqref{eq:commD1}. We can combine this with equation~\eqref{eq:ALop} and the fact that $\mathcal{M}_{j,d}^{\gamma}$ commutes with $\mathcal{M}_{d}^{\gamma}$ to deduce \eqref{eq:commD2}.
 \end{proof}

\begin{remark}\label{rem:Dar_alg}
From Remark~\ref{rem:ind} we know that
$$\mathcal{M}_{j,d}^{\gamma-k\mathbf{e}_1}=\mathcal{M}_{j,d}^{\gamma}\text{ for }j=2,\dots,d.$$
Combining this with the last Lemma, Proposition~\ref{prop:BhsMulti} and the spectral equations \eqref{speceq}, we see that
\begin{equation}\label{eq:MQ}
\mathcal{M}_{j,d}^\gamma [Q_\eta^\gamma(x)] = - |\boldsymbol{\eta}^j| (|\boldsymbol{\eta}^j| + |\boldsymbol{\gamma}^j| +d+1-j )Q_\eta^\gamma(x), \quad j=2,\dots,d.
\end{equation}
Therefore, if $\hat{\mathfrak{C}}^{\gamma;\psi}$ denotes the commutative algebra generated by the partial differential operators
$$\{{\mathcal{B}}_f(\hat{D}_1,\hat{D}_2): f\in\mathfrak{A}^{\alpha,\beta;\psi}\}\cup\{\mathcal{M}_{j,d}^\gamma, j=2,\dots,d\},$$
then equations \eqref{eq:Qeta} and \eqref{eq:MQ} show that this algebra will act diagonally on the basis $\{Q_\eta^\gamma (x),\eta\in\mathbb{N}_0^d \}$ of $\Rset[x_1,\dots,x_d]$.
Moreover, if we define $\mathfrak{C}^{\gamma;\psi}$ to be the commutative algebra  generated by the partial differential operators
$$\{f(\mathcal{M}_d^{\gamma-k \mathbf{e}_1}+\lambda_{k/2}^{-|\gamma|-d}): f\in\mathfrak{A}^{\alpha,\beta;\psi}\}\cup\{\mathcal{M}_{j,d}^\gamma, j=2,\dots,d\},$$
then Proposition~\ref{pr:multiDarboux} and Lemma~\ref{le:comm} show that
\begin{equation}
\hat{\mathfrak{C}}^{\gamma;\psi}\,\hat{\mathcal{B}}_{\psi}(\hat{D}_1,\hat{D}_{2})=\hat{\mathcal{B}}_{\psi}(\hat{D}_1,\hat{D}_{2})\,\mathfrak{C}^{\gamma;\psi},
\end{equation}
i.e. operator $\hat{\mathcal{B}}_{\psi}(\hat{D}_1,\hat{D}_{2})$ intertwines the commutative algebras  $\mathfrak{C}^{\gamma;\psi}$ and $\hat{\mathfrak{C}}^{\gamma;\psi}$.
\end{remark}

\begin{remark}\label{rem:Int2}
For $f\in\mathfrak{A}^{\alpha,\beta;\psi}$ the operator $\mathcal{B}_f(\hat{D}_1,\hat{D}_{2})$ together with the operators $\mathcal{M}_{j,d}^\gamma$, $j=2,\dots,d$, form a collection of $d$ mutually commuting and algebraically independent partial differential operators, which can be considered as a quantum completely integrable system within the general algebraic framework developed in \cite{BEG}. Note also that the commutative algebra $\hat{\mathfrak{C}}^{\gamma;\psi}$ defined in Remark~\ref{rem:Dar_alg} is not contained in any commutative algebra generated by only $d$ operators, and therefore is supercomplete  \cite{CV}. Finally, one can extend the arguments in Lemma~\ref{le:comm} and show that, for $2\leq i<j\leq d+1$, the operators $\hat{\mathcal{H}}_{i,j}$ defined in \eqref{H2}-\eqref{H3} also commute with the operators in the algebra $\hat{\mathfrak{D}}_\gamma$. This means that, for $2\leq i<j\leq d+1$, the operators $\hat{\mathcal{H}}_{i,j}$ will commute with $\mathcal{B}_f(\hat{D}_1,\hat{D}_{2})$, providing more than $d$ integrals of motion when $d>2$.  Therefore, in view of \cite{MPW} (see also Remark~\ref{rem:Int1}), we can think of this system as a quantum superintegrable system.\end{remark}

\section{Extensions of the Krall polynomials}\label{se4}

\subsection{Discrete Darboux Transformation}\label{ss4.1}
In this section, following \cite{GY02,I11}, we explain how we can pick the functions $\{\psi_n^{(0)}, \dots,\psi_n^{(k-1)} \}$, so that the polynomials $q_n^{\alpha,\beta;\psi}(t)$ defined by \eqref{q} satisfy a second-order recurrence relation in the degree index $n$.

Let $\mathcal{L}_0$ denote the bi-infinite extension of the recurrence operator $L_{\alpha,\beta}(n,E_n)$ in \eqref{eq:rec_op}
\begin{equation*}
\mathcal{L}_0=\mathcal{L}_{\alpha,\beta}(n+\varepsilon,E_n) = A_{n+\varepsilon} E_n+B_{n+\varepsilon} \Id + C_{n+\varepsilon} E_n^{-1},
\end{equation*}
which acts on functions defined on $\Zset$, where $E_n$ is the shift operator,  $A_n$, $B_n$, $C_n$ are defined in \eqref{ABC}, and $\varepsilon$ is a generic parameter such that the coefficients $A_{n+\varepsilon}$ and  $C_{n+\varepsilon}$ are well-defined and nonzero for $n\in\Zset$.
We assume below that the parameters $\alpha,\beta$ and $k$ satisfy the conditions $\alpha>-1$, $\beta\in\mathbb{N}$, and $k\leqslant\beta$. The lattice version of the elementary Darboux transformation amounts to factoring the operator $\mathcal{L}_0$ as a product of two operators and producing a new operator by exchanging the factors.  If we iterate this process $k$ times, we obtain a new operator  $\hat{\mathcal{L}}$ as follows:
\begin{align}
&\mathcal{L}_0=\mathcal{P}_0\mathcal{Q}_0 \mapsto \mathcal{L}_1:=\mathcal{Q}_0\mathcal{P}_0=\mathcal{P}_1\mathcal{Q}_1 \mapsto \dots \nonumber\\
&\qquad \mapsto \mathcal{L}_{k-1}:=\mathcal{Q}_{k-2}\mathcal{P}_{k-2}=\mathcal{P}_{k-1}\mathcal{Q}_{k-1}
\mapsto \hat{\mathcal{L}}\equiv\mathcal{L}_{k}:=\mathcal{Q}_{k-1}\mathcal{P}_{k-1}.\label{eq:dDarboux}
\end{align}
Since $\mathcal{L}_{j} \mathcal{Q}_{j-1} =\mathcal{Q}_{j-1} \mathcal{L}_{j-1}$ for each $j=1,\dots,k$, if we define $\mathcal{Q} =  \mathcal{Q}_{k-1} \dots  \mathcal{Q}_0$ it follows that
\begin{equation}\label{eq:dint}
\hat{\mathcal{L}} \mathcal{Q} = \mathcal{Q} \mathcal{L}_0,
\end{equation}
which is a discrete analog of the intertwining relation \eqref{Darbouxg}.

The sequence of Darboux transformations \eqref{eq:dDarboux} is characterized by choosing a basis $\psi_{n+\varepsilon}^{(0)},\psi_{n+\varepsilon}^{(1)},\dots, \psi_{n+\varepsilon}^{(k-1)}$ of $\ker(\mathcal{Q})$, satisfying
\begin{equation}\label{eq:basisQ}
\begin{split}
&\mathcal{L}_{0}\psi_{n+\varepsilon}^{(0)} = 0,
\\
&\mathcal{L}_{0}\psi_{n+\varepsilon}^{(j)} =  \psi_{n+\varepsilon}^{(j-1)}, \text{ for } j=1,\dots,k-1.
\end{split}
\end{equation}
To construct such a basis, we define the functions
\begin{equation}\label{phi}
\begin{split}
&\phi_n^{1,j}= \frac{(-1)^{n} (n + 1)_j (-n - \alpha - \beta)_j}{j! (1 - \beta)_j},
\\
&\phi_n^{2,j}= \frac{(-1)^{n} (n + 1)_\beta (n + \alpha + 1)_\beta (-n)_j (n + \alpha + \beta + 1)_j}{j!  \beta!(1 + \beta)_j (1 + \alpha)_\beta},
\end{split}
\end{equation}
where $j=0,1,\dots,k-1$. One can show that they are linearly independent and for fixed $i$ they satisfy equations \eqref{eq:basisQ}
\begin{align}
&\mathcal{L}_{0}\big[\phi_{n+\varepsilon}^{i,0}\big] = 0, &&\text{ for } i=1,2,
\\
&\mathcal{L}_{0}\big[\phi_{n+\varepsilon}^{i,j}\big] = \phi_{n+\varepsilon}^{i,j-1}, &&\text{ for } j=1,\dots,k-1, \text{ and } i=1,2.
\end{align}
Therefore, we can define a basis of $\ker(\mathcal{Q})$ satisfying \eqref{eq:basisQ} by setting
\begin{equation}\label{psi}
\psi_n^{(j)}\equiv\psi_n^{(j);\alpha,\beta,a} = \phi_n^{2,j} + \sum_{l=0}^j a_{j-l} \phi_n^{1,l}.
\end{equation}
Note that the basis depends on the parameters $\beta\in\mathbb{N}$, $\alpha>-1$, and $k$ free constants $a=(a_0,\dots,a_{k-1})$.
If we define new polynomials in terms of the operator $\mathcal{Q}$ as follows
\begin{equation}\label{qnew}
q_n^{\alpha,\beta;\psi}(t) = \lim_{\varepsilon\to0}\mathcal{Q} [p_{n+\varepsilon}^{\alpha,\beta}(t)] = \Wr_{n} \big(\psi_n^{(0)},\dots,\psi_n^{(k-1)}, p_n^{\alpha,\beta}(t)\big),
\end{equation}
then, using equations \eqref{eqn} and \eqref{eq:dint} and by considering the limit $\varepsilon\to 0$, we see that
$$
\hat{\mathcal{L}}(n,E_n)q_n^{\alpha,\beta;\psi}(t) = t \, q_n^{\alpha,\beta;\psi}(t).
$$
Therefore, by Favard's theorem, the polynomials $q_n^{\alpha,\beta;\psi}(t) $ are mutually orthogonal with respect to a nondegenerate moment functional.

On the other hand, note that the functions $\phi_n^{i,j}$ defined in \eqref{phi} belong to the space $(-1)^n\Rset[\lambda_n^{\alpha+\beta+1}]$. Thus, condition \eqref{eq:cond_psi} is satisfied and therefore the polynomials $q_n^{\alpha,\beta;\psi}(t)$ are also eigenfunctions of the differential operators described in Theorem~\ref{th1}. Moreover, we can use the techniques developed in Section~\ref{se3} to derive multivariable extensions of these results. In the next subsection we treat in a detail the case $k=1$.

\subsection{An explicit example}\label{ss4.2}
Let us illustrate the results and the constructions with the simplest possible example, by taking $\beta=1$ and by performing only one Darboux transformation, which means that $k=1$.
\subsubsection{One variable ingredients}
Using equations~\eqref{phi} and \eqref{psi} we see that
$$
\psi_n^{(0)} \equiv  \psi_n^{(0);\alpha,1;a_0} =  a_0 \phi_n^{1,0} +\phi_n^{2,0} = (-1)^n \left(a_0 + \frac{(n+1)(n+\alpha+1)}{\alpha+1}\right),
$$
and the polynomials $q_n^{\alpha,1;\psi} (t)$ defined in \eqref{qnew} become
$$
q_n^{\alpha,1;a_0} (t) \equiv  q_n^{\alpha,1;\psi} (t) = \Wr_n\big(\psi_n^{(0)}, p_n^{\alpha,1}(t)\big) = \begin{vmatrix} \psi_n^{(0)} & p_n^{\alpha,1}(t) \\ \psi_{n-1}^{(0)} & p_{n-1}^{\alpha,1}(t)\end{vmatrix}.
$$
One can show that these polynomials satisfy the orthogonality relations
\begin{equation}\label{eq:qorth}
\int_0^1 q_n^{\alpha,1;a_0} (t) q_m^{\alpha,1;a_0} (t) (1-t)^\alpha \mathrm{d}t + \frac{1}{a_0(\alpha+1)} q_n^{\alpha,1;a_0} (0) q_m^{\alpha,1;a_0} (0) = 0, \quad n \neq m.
\end{equation}
With $k=1$, equation \eqref{tau} tells us that
$$\tau_n=a_0 + \frac{(n + 1) (n + \alpha + 1)}{\alpha + 1}.$$
With a similar reasoning as in \cite{I11}, it can be proved that the algebra $\mathfrak{A}^{\alpha,1;a_0}$ is generated by two polynomials of degree $ 2$ and $3$, namely,
$$
\begin{aligned}
f_2(t)=&t^2-\frac{1}{2} (3+4 {a_0}+4 \alpha +4   {a_0}\alpha)t ,
\\
f_3(t)=&t^3 - \frac{1}{4} (1+6 {a_0}+6 \alpha +6   {a_0}\alpha)t^2
\\&
 -\frac{1}{16} (21+12{a_0}+28 \alpha+12 {a_0}\alpha  +4 \alpha ^2 )t.
\end{aligned}
$$
Thus, the algebra $\mathfrak{D}^{\alpha,1;a_0}$ defined in Theorem~\ref{th1} is generated by $\mathcal{B}_{f_2}$ and $\mathcal{B}_{f_3}$, which are differential operators of order $4$ and $6$, respectively. The explicit expression for $\mathcal{B}_{f_2}$ as an element of $\Rset\langle \mathcal{D}_1,\mathcal{D}_2 \rangle$ is
$$\begin{aligned}
\mathcal{B}_{f_2} =&
 \mathcal{D}_1^4-2\mathcal{D}_2\mathcal{D}_1^2+\mathcal{D}_2^2+2(1+\alpha)\mathcal{D}_1^3-2\alpha\mathcal{D}_2\mathcal{D}_1
 \\&+(1+2a_0+3\alpha+2a_0\alpha+\alpha^2)\mathcal{D}_1^2-2(1+a_0+a_0\alpha)\mathcal{D}_2
 \\&+(1+\alpha)(\alpha+2a_0(1+\alpha))\mathcal{D}_1-\frac{1}{16}(3+2\alpha)(3+6\alpha+8a_0(1+\alpha)).
\end{aligned}
$$
A similar expression can be derived for $\mathcal{B}_{f_3}$.

The operator $\Bh_{\psi}(\mathcal{D}_1,\mathcal{D}_2)$ in Proposition~\ref{prop:conn} can be written explicitly as follows
\begin{equation*}
\Bh_{\psi}(\mathcal{D}_1,\mathcal{D}_2)=\frac{1}{(1+\alpha)^2}\big(-\mathcal{D}_2 +\mathcal{D}_1^2+\alpha\mathcal{D}_1 +(1+\alpha)a_0\big),
\end{equation*}
which combined with \eqref{M1D1D2} shows that
\begin{equation}\label{Bhpsik=1}
\Bh_{\psi}(\mathcal{D}_1,\mathcal{D}_2)=-\frac{1}{(1+\alpha)^2}(\mathcal{M}_1^{\alpha,-1}  -(1+\alpha)a_0).
\end{equation}
With this formula, equation~\eqref{eq:conn} reads
\begin{equation*}
q_n^{\alpha,1;a_0} (t)=(-1)^n(2n+\alpha+1)\Bh_{\psi}(\mathcal{D}_1,\mathcal{D}_2)[p_n^{\alpha,0}(t)].
\end{equation*}
The polynomial depending on the parameter $s$, defined in \eqref{qcup} becomes
$$
\hat{q}_{n,s}^{\alpha,1;a_0} (t)= \Wr_n\big(\psi_{n+s}^{(0)}, p_n^{\alpha+2s,1}(t)\big),
$$
and Proposition~\ref{prop:Bhs} yields
\begin{equation*}
\hat{q}_{n,s}^{\alpha,1;a_0} (t)(1-t)^s=(-1)^{n+s}\frac{(2n+2s+\alpha+1)(\alpha+1)}{\alpha+1+2s}\,\Bh_{\psi}(\mathcal{D}_1,\mathcal{D}_{2,s})[p_n^{\alpha,0}(t)(1-t)^s].
\end{equation*}
Observe that
$$
\psi_{n+s}^{(0)} \equiv \psi_{n+s}^{(0);\alpha,1;a_0} = (-1)^s \,\frac{\alpha+2s+1}{\alpha+1} \, \psi_{n}^{(0);\alpha+2s,1;a_0^s},
$$
with $\displaystyle a_0^s=\frac{a_0(1+\alpha)+s(s+\alpha)}{\alpha+2s+1}$, and therefore
\begin{equation*}
\hat{q}_{n,s}^{\alpha,1;a_0}(t)=(-1)^s \, \frac{\alpha+2s+1}{\alpha+1} \, q_n^{\alpha+2s,1;a_0^s}(t).
\end{equation*}
Combining this with \eqref{eq:qorth}, we see that the polynomials $\hat{q}_{n,s}^{\alpha,1;a_0}(t)$ satisfy the orthogonality relation
\begin{equation}\label{orthoqhat}
\begin{aligned}
\left(1+\frac{s(\alpha+s)}{a_0(\alpha+1)}\right) \int_0^1 & \hat{q}_{n,s}^{\alpha,1;a_0}(t)  \hat{q}_{m,s}^{\alpha,1;a_0}(t) (1-t)^{\alpha+2s} \mathrm{d}t
\\
+ & \frac{1}{a_0(\alpha+1)} \hat{q}_{n,s}^{\alpha,1;a_0}(0) \hat{q}_{m,s}^{\alpha,1;a_0}(0) = 0 , \quad n \neq m.
\end{aligned}
\end{equation}

\subsubsection{Multivariable operators and polynomials}

Now, let us take $\gamma=(\gamma_1,\dots,\gamma_{d+1})$ with $\gamma_1=1$, and consider the polynomials  $Q_\eta^\gamma(x)$ defined as in \eqref{Qeta}:
\begin{equation}\label{qex}
Q_\eta^\gamma(x) = \hat{q}_{\eta_1,s}^{\alpha,\beta;a_0} (z_1) (1-z_1)^s P_{\eta_2,\dots,\eta_d}^{\gamma_2,\dots,\gamma_{d+1}}(z_2,\dots,z_d),
\end{equation}
where the variables $x$ and $z$ are related as in \eqref{chvar} and
$$
\begin{aligned}
&s=\eta_2+\dots+\eta_d,
\\
&\alpha=\gamma_2+\dots+\gamma_{d+1}+d-1,
\\
&\beta=\gamma_1=1.
\end{aligned}
$$
The partial differential operator $\hat{\mathcal{B}}_{\psi}(\hat{D}_1,\hat{D}_2)$ can be computed from $\Bh_{\psi}(\mathcal{D}_1,\mathcal{D}_2)$ in \eqref{Bhpsik=1} by applying the map \eqref{map}. With the fixed values of the parameters in this section and applying \eqref{eq:JtoAL} we see that
\begin{equation}\label{Bhpsimv}
\hat{\mathcal{B}}_{\psi}(\hat{D}_1,\hat{D}_2)=-\frac{1}{(|\boldsymbol\gamma^2|+d)^2}(\mathcal{M}_d^{\tilde{\gamma}}
-(|\boldsymbol\gamma^2|+d)a_0),
\end{equation}
where $\mathcal{M}_d^{\tilde{\gamma}} $ is the Appell-Lauricella operator \eqref{diffeq1} with parameters
$$\tilde{\gamma}=(-1,\gamma_2,\gamma_3,\dots,\gamma_{d+1}).$$
Proposition~\ref{prop:BhsMulti} yields
$$Q_\eta^\gamma(x)=\frac{(-1)^{|\eta|+1}(2|\eta|+|\boldsymbol\gamma^2|+d)}{(|\boldsymbol\gamma^2|+d)(|\boldsymbol\gamma^2|+d+2|\boldsymbol\eta^2|)}(\mathcal{M}_d^{\tilde{\gamma}}  -(|\boldsymbol\gamma^2|+d)a_0)[P_\eta^{0,\gamma_2,\gamma_3,\dots,\gamma_{d+1}}(x)].$$

With the above constructions, we know that the spectral equations \eqref{eq:Qeta} and \eqref{eq:MQ} hold. Finally, we show next that the polynomials $Q_\eta^{\gamma}(x)$ satisfy a generalized orthogonality relation which resembles the Sobolev inner products studied in the literature, see the review article \cite{MX2015} and the references therein.

\begin{theorem}
The polynomials \eqref{qex} are mutually orthogonal with respect to the Sobolev inner product
\begin{align}
\langle f,g\rangle  = & \int_{\mathbb{T}^d} f(x)\,g(x)\,x_2^{\gamma_2}\cdots x_d^{\gamma_d} (1-|x|)^{\gamma_{d+1}}\mathrm{d}x
\nonumber\\
&+\frac{1}{a_0 (|\boldsymbol\gamma^2|+d)} \int_{\mathbb{T}^{d-1}} f(0,x_2,\dots,x_d) \, g(0,x_2,\dots,x_d)\nonumber\\
&\hspace{3cm}\times
x_2^{\gamma_2}\cdots x_{d}^{\gamma_d} (1-x_2-\cdots-x_d)^{\gamma_{d+1}}\,\mathrm{d}x_2\cdots \mathrm{d}x_{d}\nonumber\\
&
- \frac{1}{a_0 (|\boldsymbol\gamma^2|+d)} \int_{\mathbb{T}^d} \mathcal{M}_{2,d}^{\gamma}[f(x)] \, g(x) \, x_2^{\gamma_2}\cdots x_d^{\gamma_d} (1-|x|)^{\gamma_{d+1}}\,\mathrm{d}x.\label{eq:Sorth}
\end{align}
\end{theorem}
\begin{proof}
Let us take $\eta\neq\xi\in\mathbb{N}^d$, $s=\eta_2+\dots+\eta_d$, and $\bar{s}=\xi_2+\dots+\xi_d$.
We want to show that
\begin{equation}\label{eq:pr_morth}
\langle Q_\eta^{\gamma}(x),Q_\xi^{\gamma}(x)\rangle=0.
\end{equation}
From \eqref{eq:MQ} with $j=2$ we know that
$$
\mathcal{M}_{2,d}^{\gamma}(x)[Q_\eta^{\gamma}(x)]=\lambda_s^{|\boldsymbol\gamma^2|+d-1}\, Q_\eta^{\gamma}(x),
$$
where $\lambda_s^{|\boldsymbol\gamma^2|+d-1} = -s(s+|\boldsymbol\gamma^2|+d-1)=-s(s+\alpha)$.
Substituting this into~\eqref{eq:Sorth} and using the definition~\eqref{Qeta} of the polynomials $Q_\eta^{\gamma}(x)$ we can rewrite the inner product as follows:

\begin{align}
\langle Q_\eta^{\gamma}(x)&,Q_\xi^{\gamma}(x)\rangle\nonumber\\
&=
 \int_{\mathbb{T}^{d-1}}  P_{\eta_2,\dots,\eta_d}^{\gamma_2,\dots,\gamma_{d+1}}(z) P_{\xi_2,\dots,\xi_d}^{\gamma_2,\dots,\gamma_{d+1}}(z) z_2^{\gamma_2}\dots z_d^{\gamma_d}(1-|z|)^{\gamma_{d+1}}\,\mathrm{d}z_2\cdots \mathrm{d}z_{d}\nonumber\\
&\times\Bigg(\left(1-\frac{\lambda_s^{|\boldsymbol\gamma^2|+d-1}}{a_0(\alpha+1)} \right)
\int_0^1 \hat{q}_{\eta_1,s}^{\alpha,1;a_0} (z_1) \hat{q}_{\xi_1,\bar{s}}^{\alpha,1;a_0} (z_1) (1-z_1)^{|\boldsymbol\gamma^2|+s+\bar{s}+d-1} \mathrm{d}z_1
\nonumber\\
&\qquad\qquad + \frac{1}{a_0(\alpha+1)}\hat{q}_{\eta_1,s}^{\alpha,1;a_0} (0) \hat{q}_{\xi_1,\bar{s}}^{\alpha,1;a_0} (0)\Bigg).\label{Qetaorth}
\end{align}
There are two possible cases:
\begin{itemize}
\item If $(\eta_2,\dots,\eta_d)\neq(\xi_2,\dots,\xi_d)$, the orthogonality of the simplex polynomials $P_{\eta_2,\dots,\eta_d}^{\gamma_2,\dots,\gamma_{d+1}}(z)$ shows that the second line of \eqref{Qetaorth} is $0$, proving \eqref{eq:pr_morth}.

\item If we assume that $(\eta_2,\dots,\eta_d)=(\xi_2,\dots,\xi_d)$, but $\eta_1\neq\xi_1$, then $\bar{s}=s$. Therefore, $|\boldsymbol\gamma^2|+s+\bar{s}+d-1=|\boldsymbol\gamma^2|+d-1+2s=\alpha+2s$ which shows that the factor in the last two lines of \eqref{Qetaorth} is $0$ by \eqref{orthoqhat}, completing the proof.
\end{itemize}
\end{proof}


\section*{Acknowledgments} We would like to thank a referee for a careful reading of the manuscript and helpful suggestions.


\end{document}